\documentclass[11pt]{amsart}
\usepackage{amsmath, amsfonts, amsbsy, amssymb,  enumerate, upref}

\newtheorem{thm}{Theorem}[section]
\newtheorem{lmm}[thm]{Lemma}

\newtheorem{prop}[thm]{Proposition}

\newcommand{\argmax}{\operatorname{argmax}}

\newcommand{\ee}{\mathbb{E}}

\newcommand{\pp}{\mathbb{P}}

\newcommand{\ra}{\rightarrow}
\newcommand{\rr}{\mathbb{R}}
\newcommand{\smallavg}[1]{\langle #1 \rangle}

\newcommand{\ep}{\epsilon}

\newcommand{\mm}{\mathcal{W}}

\newcommand{\mmm}{\widetilde {\mathcal W}}
\newcommand{\wf}{\widetilde{f}}
\newcommand{\WF}{{\widetilde F}}
\newcommand{\wh}{{\widetilde h}}

\begin{document}
\title[Large deviations for random graphs]{The large deviation principle for the Erd\H{o}s-R\'enyi random graph}
\author{Sourav Chatterjee}
\address{Courant Institute of Mathematical Sciences, New York University, 251 Mercer Street, New York, NY 10012}
\thanks{Sourav Chatterjee's research was partially supported by NSF grants DMS-0707054 and  DMS-1005312,   and a Sloan Research Fellowship}
\author{S. R. S. Varadhan}
\thanks{S.~R.~S.~Varadhan's research was partially supported by NSF grants DMS-0904701 and OISE-0730136.}

\begin{abstract}
What does an Erd\H{o}s-R\'enyi graph look like when a rare event happens? This paper answers this question when $p$ is fixed and $n$ tends to infinity by establishing a large deviation principle under an appropriate topology. The formulation and proof of the main result uses the recent development of the theory of graph limits by Lov\'asz and coauthors and  Szemer\'edi's regularity lemma from graph theory. As a basic application of the general principle, we work out large deviations for the number of triangles in $G(n,p)$. Surprisingly, even this simple example yields an  interesting double phase transition. 
\end{abstract}
\maketitle

\section{Introduction}\label{intro}
\subsection{The Erd\H{o}s-R\'enyi graph}
Let $G(n,p)$ be the random graph on $n$ vertices where each edge is added independently with probability $p$. This model has been the subject of extensive investigations since the pioneering work of Erd\H{o}s and R\'enyi \cite{erdosrenyi60}, yielding a  large body of literature (see \cite{bollobas01, JLR00} for partial surveys). 

This paper studies the following basic aspect of Erd\H{o}s-R\'enyi graphs: What does the graph look like if one knows that some rare event has happened? One way to comprehensively answer this question is to formulate a large deviation principle for the Erd\H{o}s-R\'enyi graph, in the same way as Sanov's theorem \cite{sanov61} gives a large deviation principle for an i.i.d.\ sample. 

The setting of Sanov's theorem conforms naturally to the abstract theory of large deviations (see Chapter 6 in \cite{dembozeitouni98}) because i.i.d.\ samples can be thought of as random probability measures, allowing them to be viewed as random elements of a single topological space irrespective of the sample size. The first hurdle in formulating such a program for random graphs is in constructing a single abstract space in which all graphs can be embedded. Fortunately, this issue has been settled recently. In a sequence of papers \cite{borgsetal06, borgsetal08, borgsetal07, freedmanlovaszschrijver07, lovasz06, lovasz07, lovaszsos08, lovaszszegedy06, lovaszszegedy07, lovaszszegedy07b, lovaszszegedy09} Laszlo 
Lov\'asz with coauthors (listed here in order of frequency) V.~T.~S\'os, B.~Szegedy, C.~Borgs, J.~Chayes, K.~Vesztergombi, A.~Schrijver and M.~Freedman have developed a beautiful, unifying limit theory. (See also the related work of Diaconis and Janson \cite{diaconisjanson08} which traces this back to work of Aldous \cite{aldous81} and Hoover~\cite{hoover82}.) This sheds light on topics such as graph homomorphisms, Szemer\'edi's regularity lemma, quasi-random graphs, graph testing and extremal graph theory, and has even found applications in statistics and related areas (see e.g.~\cite{cds10}). Their theory has been developed for dense graphs (number of edges comparable with the square of number of vertices) but parallel 
theories for sparse graphs are beginning to emerge~\cite{bollobasriordan09}. 

\subsection{Graph limits and graphons}
The limit of a sequence of dense graphs can be defined as follows. We quote the definition verbatim from \cite{lovaszszegedy06} (see also \cite{borgsetal08, borgsetal07, diaconisjanson08}). Let $G_n$ be a sequence of simple graphs whose number of nodes tends to infinity. For every fixed simple graph $H$, let $|\hom(H, G)|$ denote the number of homomorphisms of $H$ into $G$ (i.e.\ edge-preserving maps $V(H) \ra V(G)$, where $V(H)$ and $V(G)$ are the vertex sets). This number is normalized to get the homomorphism density 
\begin{equation}\label{homdens}
t(H,G) := \frac{|\hom(H, G)|}{|V(G)|^{|V(H)|}}. 
\end{equation}
This gives the probability that a random mapping $V(H) \ra V(G)$ is a homomorphism. 

Suppose that the graphs $G_n$ become more and more similar in the sense that $t(H, G_n)$ tends to a limit $t(H)$ for every $H$. One way to define a limit of the sequence $\{G_n\}$ is to define an appropriate limit object from which the values $t(H)$ can be read off. 

The main result of \cite{lovaszszegedy06} (following the earlier equivalent work of Aldous~\cite{aldous81} and Hoover \cite{hoover82}) is that indeed there is a natural ``limit object'' in the form of a function $f\in \mm$, where $\mm$ is the space of all measurable functions from $[0,1]^2$ into $[0,1]$ that satisfy $f(x,y)=f(y,x)$ for all $x,y$.

Conversely, every such function arises as the limit of an appropriate graph sequence. This limit object determines all the limits of subgraph densities: if $H$ is a simple graph with $V(H) = [k] = 
\{1, \ldots, k\}$, then 
\begin{equation}\label{tfdef}
t(H,f) = \int_{[0,1]^k}\prod_{(i,j)\in E(H)} f(x_i, x_j) \;dx_1\cdots dx_k. 
\end{equation}
Here $E(H)$ denotes the edge set of $H$.  A sequence of graphs $\{G_n\}_{n\ge 1}$ is said to converge to $f$ if for every finite simple graph $H$,
\begin{equation}\label{gconv}
\lim_{n\ra \infty} t(H, G_n) = t(H,f).
\end{equation}
Intuitively, the interval $[0,1]$ represents a `continuum' of vertices, and $f(x,y)$ denotes the probability of putting an edge between $x$ and $y$. For example, for the Erd\H{o}s-R\'enyi graph $G(n,p)$, if $p$ is fixed and $n \ra \infty$, then the limit graph is represented by the function that is identically equal to $p$ on~$[0,1]^2$.

These limit objects, i.e.\ elements of $\mm$, are called `graphons' in \cite{lovaszszegedy06, borgsetal08, borgsetal07}. A finite simple graph $G$ on $\{1,\ldots,n\}$ can also be represented as a graphon $f^G$ is a natural way, by defining
\begin{equation}\label{wg}
f^G(x,y) = 
\begin{cases}
1 &\text{ if $(\lceil nx\rceil, \lceil ny\rceil)$ is an edge in $G$,}\\
0 &\text{ otherwise.} 
\end{cases}
\end{equation}
Note that this allows {\it all} simple graphs, irrespective of the number of vertices, to be represented as elements of a single abstract space, namely $\mm$. 

\subsection{The cut metric}% and compactness of graphon space}
With the above representation, it turns out that the notion of convergence in terms of subgraph densities outlined above can be captured by an explicit metric on $\mm$, the so-called `cut distance' (originally defined for finite graphs by Frieze and Kannan \cite{friezekannan99}). We start with the space $\mm$ of measurable functions $f(x,y)$  on $[0,1]^2$ that  satisfy $0\le f(x,y)\le 1$ and $f(x,y)=f(y,x)$. We define the cut distance 
\begin{equation}\label{defcut}
d_\square (f,g) := \sup_{S,T\subseteq [0,1]} \biggl|\int_{S\times T} [f(x,y)-g(x,y)] dx dy\biggr|.
\end{equation}
We introduce in $\mm$ an equivalence relation. Let $\Sigma$ be the space of  measure preserving bijections $\sigma:[0,1]\ra[0,1]$. Say that $f(x,y)\sim g(x,y)$ if $f(x,y)=g_\sigma(x,y):=g(\sigma x, \sigma y)$ for some $\sigma\in\Sigma$. Denote by ${\widetilde g}$ the closure  in $(\mm, d_\Box)$ of  the  orbit $\{g_\sigma\}$. The quotient space is  denoted by $\mmm$ and $\tau$ denotes the natural map $g\to{\widetilde g}$. Since  $d_\Box$ is invariant under $\sigma$ one can define on $\mmm$, the natural distance $\delta_\Box$ by
$$
\delta_\Box({\widetilde f},{\widetilde g}):=\inf_\sigma d_\Box(f, g_\sigma)=\inf_\sigma d_\Box(f_\sigma, g)=\inf_{\sigma_1,\sigma_2}d_\Box(f_{\sigma_1}, g_{\sigma_2})
$$
making $(\mmm, \delta_\Box)$ into a metric space. To any  finite graph $G$, we associate $f^G$ as in \eqref{wg} and its orbit  ${\widetilde G}=\tau f^G= {\widetilde f}^G\in\mmm$. One of the key results of \cite{borgsetal08} is the following:
\begin{thm}[Theorem 3.8 in \cite{borgsetal08}]
A sequence of graphs $\{G_n\}_{n\ge 1}$ converges to a limit $f\in \mm$ in the sense defined in \eqref{gconv} if and only if $\delta_\Box({\widetilde G}_n, {\widetilde f}) \ra 0$ as $n \ra \infty$.
\end{thm} 
Szemer\'{e}di's regularity lemma and the related deep  results of Lov\'asz and Szegedy will play a crucial role in  this paper:
\subsection{Szemer\'edi's lemma}

Let $G= (V,E)$ be a simple graph, and let $X, Y$ be subsets of $V$. Then we denote by $e_G(X,Y)$ the number of $X$-$Y$ edges of $G$ (edges whose endpoints belong to $X\cap Y$ are counted twice), and call
\[
\rho_G(X,Y) := \frac{e_G(X,Y)}{|X||Y|}
\]
the {\it density} of the pair $(X,Y)$. Given some $\ep > 0$, we call a pair $(A,B)$ of disjoint sets $A, B \subseteq V$ $\ep$-regular if all $X\subseteq A$ and $Y\subseteq B$ with $|X|\ge \ep|A|$ and $|Y|\ge \ep |B|$ satisfy
\[
|\rho_G(X,Y) - \rho_G(A,B)| \le \ep. 
\]
A partition $\{V_0,\ldots,V_K\}$ of $V$ is called an $\ep$-regular partition of $G$ if it satisfies the following two conditions:
\begin{enumerate}
\item [(i)] $|V_0|\le \epsilon n$;
\item[(ii)] $|V_1|=|V_2|=\cdots= |V_K|$;
\item[(iii)] all but at most $\ep K^2$ of the pairs $(V_i, V_j)$ with $1\le i<j\le K$ are $\ep$-regular. 
\end{enumerate}
Szemer\'edi's regularity lemma goes as follows. 
\begin{thm}[Szemer\'edi's lemma \cite{szemeredi78}]
Given  $\ep >0$ and  an integer $m \ge 1$ there exists an integer $M = M(\ep,m)$ such that every graph  of order at least $M$   admits an $\ep$-regular partition $\{V_0,\ldots, V_K\}$ for some $K$ in the range  $m\le K\le M$. 
\end{thm}
\noindent This result was proved by Szemer\'edi \cite{szemeredi78} in 1976 and has since found numerous applications in combinatorics, number theory and many other areas of discrete mathematics. The version presented above is from Diestel~\cite{diestel00}, Section 7.2.  Lov\'asz and Szegedy proved the following related result.
\begin{thm}[Theorem 5.1 in \cite{lovaszszegedy07}]\label{compact}
The metric space $(\mmm, \delta_\Box)$ is compact.
\end{thm}
\section{The main result}
\subsection{The rate function}
The main goal of this paper is to prove a large deviation principle for $G(n,p)$ when $p$ is fixed and $n\ra\infty$. The discussion in Section \ref{intro} gives a topological space (namely, $\mmm$) that is suitable for this purpose. The next step is to define a rate function on this space. Let $I_p:[0,1] \ra \rr$ be the function 
\begin{align}\label{ipdef1}
I_p(u) :&= \frac{1}{2}u\log \frac{u}{p}+\frac{1}{2}(1-u)\log\frac{1-u} {1-p}\\
&=\frac{1}{2}\sup_{a,b\in \rr}\Big[ au+b(1-u)-\log\big(pe^a+(1-p)e^b\big)\Big]\nonumber\\
&=\frac{1}{2}\sup_{a\in \rr}\Big[ au-\log\big(pe^a+(1-p)\big)\nonumber
\Big]
\end{align}
The domain of the function $I_p$ can be extended to $\mm$ as 
\begin{align}
I_p(h):&=\int_0^1\int_0^1I_p(h(x,y))\,dx\,dy\nonumber\\
&=\frac{1}{2}\int_0^1\int_0^1 \biggl[h(x,y)\log \frac{h(x,y)}{p}+(1-h(x,y))\log\frac{1-h(x,y)} {1-p}\biggr]dxdy\label{ipdef2}\\
&=\frac{1}{2}\sup_{a(\cdot,\cdot)}\biggl[\int a(x,y)h(x,y)\,dx\,dy\label{ipdef3}\\
&\qquad\qquad \qquad\qquad -\int \log(pe^{a(x,y)}+(1-p))\,dx\,dy\biggr].\nonumber
\end{align}
The following property of $I_p$ is crucial. 
\begin{lmm}\label{lower}
The function $I_p$ is well defined on $\mmm$ and  is lower semicontinuous under the  cut metric $\delta_\Box$ on $\mmm$. 
\end{lmm}
\begin{proof}
The supremum in \eqref{ipdef3} can be taken over all bounded  measurable functions $a$ on $[0,1]^2$. As the supremum of a family of  affine linear functionals continuous in the weak topology, $I_p(h)$ is lower semi-continuous in the weak topology and  therefore also in the  topology of the metric $d_\square$ . If $\sigma:[0,1]\ra [0,1]$ is a measure preserving bijection then $I_p(h_\sigma)=I(h)$. By lower semi-continuity of $g\in {\widetilde h}$, 
$I_p(g)\le I_p(h)$. But $g\in {\widetilde h}$ implies $h\in {\widetilde g}$ so that $I_p(h)\le I_p(g)$. Hence, $I_p(g)=I_p(h)$
and $I_p(\cdot)$ is well defined and lower semi-continuous on $\mmm$. 
 \end{proof}
\subsection{The Large Deviation Principle}  
The random graph $G(n,p)$ induces  probability distributions $\pp_{n,p}$ on the space $\mm$ through the
map $G\ra f^G$ and ${\widetilde\pp}_{n,p}$ on $\mmm$ through the  map $G\ra f^G\ra {\widetilde f}^G= \tau f^G$. The space $\mm$
is compact in the weak topology and a large deviation principle for $\pp_{n,p}$ on $\mm$ in the weak topology
with the lower-semicontinuous rate function
$I_p(h)$ given by \eqref{ipdef2}
is elementary but is not of much use since  quantities like `triangle counts' are not stable in the weak topology. We will state it for the record and find a use for it later.
\begin{thm}\label{weak}
The sequence $\pp_{n,p}$ on $\mm$ satisfies a large deviation principle in the weak topology. That is,
for every weakly closed set $F\subset\mm$
$$
\limsup_{n\ra \infty}\frac{1}{n^2}\log \pp_{n,p}(F)\le -\inf_{f\in F} I_p(f)
$$
and for any open set $U$ (again in the weak topology)  in $\mm$
$$
\liminf_{n\ra \infty}\frac{1}{n^2}\log \pp_{n,p}(U)\ge -\inf_{f\in U} I_p(f).
$$
\end{thm}
\begin{proof}
The weak topology is defined through an arbitrary but finite number of linear functionals. Therefore the large deviation principle
can be reduced to the large deviation behavior of a finite set of linear functionals $\{Z_\phi(f)\}$ given by
$$
Z_\phi(f)=\smallavg{\phi,f} := \iint f(x,y)\phi(x,y) dx dy
$$
under the measure $\pp^{n,p}$.  
The limit  
$$
\lim_{n\ra\infty}\frac{1}{n^2}\log \ee^{\pp_{n,p}}\biggl[\exp\biggl(n^2\iint f(x,y)\phi (x,y) dx dy\biggr)\biggr]
$$
is easily calculated to yield
$$
\frac{1}{2}\iint \log (pe^{2\phi (x,y)}+(1-p)) dxdy.
$$
(Note that this is true only if $\phi$ is symmetric. However, since $f$ is symmetric, it suffices to restrict attention to symmetric $\phi$.) Therefore, an abstract G\"artner-Ellis Theorem (see e.g.\ Theorem 4.5.3 in \cite{dembozeitouni98}) gives the upper bound with rate function
$$
I_p(f)=\sup_\phi\biggl[\smallavg{\phi, f}-\frac{1}{2}\iint \log (pe^{2\phi (x,y)}+(1-p)) dxdy\biggr].
$$
Note that this is the rate function $I_p$ defined in \eqref{ipdef2} and \eqref{ipdef3}. The supremum is attained at the function
\[
\phi_f(x,y) := \frac{1}{2}\log\frac{f(x,y)}{p} -  \frac{1}{2}\log\frac{1-f(x,y)}{1-p}. 
\]
Note that for any $g\ne f$, 
\begin{align*}
&\big(\smallavg{\phi_f, f} - I_p(f)\big)-\big(\smallavg{\phi_f, g} - I_p(g)\big) \\
&= \frac{1}{2}\iint \biggl(g(x,y)\log \frac{g(x,y)}{f(x,y)} + (1-g(x,y)) \log \frac{1-g(x,y)}{1-f(x,y)}\biggr) \, dx\, dy > 0. 
\end{align*}
This shows that every $f$ is an exposed point of the lower semicontinuous rate function $I_p$, in the parlance of convex analysis. Therefore by the G\"artner-Ellis Theorem (see e.g.\ Theorem 4.5.20 in \cite{dembozeitouni98}) and the compactness of the weak topology, we get the lower bound. 
\end{proof}
The large deviation principle for ${\widetilde\pp}_{n,p}$ on $(\mmm, \delta_\Box)$ is much more useful and is the main result of this article.
\begin{thm}\label{main}
For each fixed $p\in (0,1)$, the sequence ${\widetilde \pp}_{n,p}$ obeys a large deviation principle in the space $\mmm$ (equipped with the cut metric) with rate function $I_p$ defined by \eqref{ipdef2}. Explicitly, this means that
for any closed set $\widetilde {F} \subseteq\mmm$,
\begin{align}\label{closed}
\limsup_{n\ra\infty} \frac{1}{n^2}\log {\widetilde\pp}_{n,p}(\widetilde {F}) &\le -\inf_{{\widetilde h}\in \widetilde{F}} I_p({\widetilde h}).
\end{align}
and for any open set $\widetilde{U}\subseteq \mmm$,
\begin{align}\label{open}
 \liminf_{n\ra\infty} \frac{1}{n^2}\log {\widetilde\pp}_{n,p}( \widetilde{U}) &\ge -\inf_{{\widetilde h}\in \widetilde{U}} I_p({\widetilde h}). 
\end{align}
\end{thm}
For the upper bound, because $(\mmm, \delta_\square)$ is compact,
it is sufficient to prove that for any ${\widetilde h}\in\mmm$,
$$
\lim_{\eta\ra 0}\limsup_{n\ra\infty} \frac{1}{n^2}\log{\widetilde\pp}_{n,p}(S_\square({\widetilde h},\eta)) \le - I_p({\widetilde h}).
$$
For the lower bound we need to prove that if ${\widetilde h}\in\mmm$ and $\eta>0$ is arbitrary
$$
\liminf_{n\ra\infty} \frac{1}{n^2}\log {\widetilde\pp}_{n,p}(S_\square({\widetilde h},\eta)) \ge - I_p({\widetilde h}), 
$$
where $S_\square({\widetilde h},\eta)=\{{\widetilde g}: \delta_\square({\widetilde g},{\widetilde h})\le \eta\}$.

\subsection{Proof of the upper bound in Theorem \ref{main}}

Let $B({\widetilde h}, \eta)\subset\mm$ be defined as
$$B({\widetilde h}, \eta)=\tau^{-1}S_\square({\widetilde h},\eta) \subset \mm$$
i.e. the union of all the orbits from $ S_\square({\widetilde h},\eta)\subset\mmm$.  We need   to show that
\begin{equation}\label{upperbound}
\lim_{\eta\ra 0} \limsup_{n\ra\infty} \frac{1}{n^2}\log\pp_{n,p}[B({\widetilde h},\eta)] \le - I_p({\widetilde h})
\end{equation}
Let the set of $n$ vertices be partitioned into $K$ subsets of size $a$ with a remainder of size $b$, so that $n=Ka+b$. We assume that $b\le \ep n$.  We order the vertices so that $V_0=\{1,2,\ldots b\}$ and   $V_i=\{ b+(i-1)a+1,\ldots, b+ia\}$ for $i=1,2,\ldots, K$. We map the vertices into subintervals of the  unit interval, with  the vertex $r$ represented by the interval $[\frac{r-1}{n}, \frac{r}{n}]$. The sets $V_i$ of vertices will then correspond to the intervals $E_0=[0,\frac{b}{n}]$
for $i=0$ and $E_i=[\frac{b+(i-1)a}{n},\frac{b+ia}{n}]$ for $1\le i\le K$.  Let us denote by ${\mathcal V}_{K}$ the subset of $\mm$ consisting of $g(x,y)$ that are  of the form
$$
g(x,y)=\sum_{i,j=1}^K p_{i,j}{\bf 1}_{E_i}(x){\bf 1}_{E_j}(y)
$$
where  $\{p_{i,j}\}$, $1\le i,j\le K$  is symmetric and satisfies $0\le p_{i,j}\le 1$. For any pair  $m,M$ with $m<M$,  we define ${\mathcal V}_{m,M}=\cup_{m\le K\le M}{\mathcal V}_{K}$. The following is a restatement of  the Szemer\'{e}di regularity lemma. 

\begin{lmm}\label{nl1}
Given any $\epsilon>0$  and $m\ge 1$ such that $2m^{-1}< \ep$, there is $M=M(\epsilon,m)$ such that for any graph $G$, there exists a permutation $\pi$, i.e.\ a relabeling of the vertices of the graph, such that
$$
\inf_{g\in {\mathcal V}_{m,M} }d_\Box(f^{\pi G}, g)\le \epsilon.
$$
\end{lmm}
\begin{proof}
Let $\ep'$ and $m$ be given. According to Szemer\'{e}di's lemma,  there is $M(\ep', m)$ such that,  for some $K$ in the range $m\le K\le M$,  we can find a partition $V_0,\ldots V_K$ which is $\ep'$-regular. After a permutation we can assume that the ordering of the vertices coincides with the ordering of the partitions. We define $p_{i,i}=p_{0,i}=p_{i,0}=0$  and for $1\le i\not= j$,  $p_{i,j}=\rho_G (V_i,V_j)$. This leads to 
$$
g(x,y)=\sum_{\substack{i,j=1\\ i\not=j}}^K \rho_G (V_i,V_j){\bf 1}_{E_i}(x){\bf 1}_{E_j}(y)
$$
to be compared with $f^{G}$ when $V_0,\ldots, V_K$ is an  $ \ep'$ regular partition of $G$. Recall that 
\[
d_\Box(f^G, g) =\sup_{S,T\subset [0,1]} \biggl| \iint_{S\times T}[f^G(x,y)-g(x,y)]dx\,dy   \biggr|. 
\]
Since both $f^G$ and $g$ are constant on sets of the form $[\frac{i}{n},\frac{i+1}{n}]\times [\frac{j}{n},\frac{j+1}{n}]$ it is easy to see that $S$ and $T$ can be restricted to unions of intervals  of the form $[\frac{i}{n},\frac{i+1}{n}]$ i.e.
subsets of $[0,1]$ that represent subsets of  vertices. These subsets will also be denoted by $S$ and $T$. Now, given two such subsets $S$ and $T$, 
\begin{align*}
& \iint_{S\times T}(f^G(x,y)-g(x,y))dx\,dy  \\
& = \sum_{0\le i,j\le K}   \iint_{(S\cap E_i)\times (T\cap E_j)}(f^G(x,y)-g(x,y))dx\,dy.  %\\
%&= \sum_{0\le i, j\le K} \biggl(\frac{1}{n^2}\,e_G(S\cap V_i, \, T\cap V_j) -\rho_G(V_i,V_j)|S\cap E_i||T\cap E_j|\biggr) \\
%&= \sum_{0\le i,j\le K} (\rho_G(S\cap V_i, S\cap V_j) - \rho_G(V_i, V_j)) |S\cap E_i||T\cap E_j|. 
\end{align*}
Let $A_{ij}$ denote the $(i,j)$th term in the above sum. 
Since $f^G$ and $g$ both take values in $[0,1]$, therefore for each $i$ and each~$j$,
\[
|A_{i0}| \le |E_i||E_0|, \ \ |A_{0j}|\le |E_0||E_j|.
\]
For the same reason, if $|S\cap E_i| < \ep' |E_i|$ or $|T\cap E_j|< \ep' |E_j|$, then 
\[
|A_{ij}| \le\ep'  |E_i||E_j|.
\]
If $1\le i\ne j$, $|S\cap E_i|\ge \ep' |E_i|$, $|T\cap E_j|\ge \ep' |E_j|$ and the pair $(V_i, V_j)$ is $\ep' $-regular, then 
\begin{align*}
|A_{ij}|&= \biggl|\frac{1}{n^2}\,e_G(S\cap V_i, \, T\cap V_j) -\rho_G(V_i,V_j)|S\cap E_i||T\cap E_j| \biggr|\\
&\le|\rho_G(S\cap V_i, S\cap V_j) - \rho_G(V_i, V_j)||S\cap E_i||T\cap E_j|\\
&\le \ep'|E_i||E_j|. 
\end{align*}
Finally, if either $1\le i=j$ or $(V_i, V_j)$ is not $\ep'$-regular (of which there are at most $K +2\ep' K^2$ cases), then we have the trivial bound $|A_{ij}|\le \frac{a^2}{n^2} \le \frac{1}{K^2}$. A combination of the above estimates gives
\begin{align*}
\sum_{0\le i,j\le K} |A_{ij}|&\le  2|E_0| +\ep' + (K+ 2\ep' K^2) \frac{1}{K^2}\\
&\le (2\ep' +2\ep'+\ep'+K^{-1})
\end{align*}
Thus, 
\[
d_\Box(f^G, g) \le 5\ep'  + K^{-1}\le 5\ep' + m^{-1}
\]
Since $m^{-1} < \ep/2$, we can choose $\ep'$ so that $5\ep' + m^{-1}<\ep$. 
\end{proof}
\begin{lmm}\label{nl2}
Let $\epsilon,m$ and $M$ be as in Lemma \ref{nl1}.
\begin{align*}
\pp_{n,p}(B({\widetilde h},\eta)) \le n! \,\pp_{n,p}(B({\widetilde h},\eta)\cap B({\mathcal V}_{m,M},\epsilon))
\end{align*}
where $B({\mathcal V}_{m,M},\epsilon) =\{g:  \inf_{f\in {\mathcal V}_{m,M}} d_\Box(g,f)\le \epsilon\}$.
\end{lmm}
\begin{proof}
The orbit under the permutation group has at most $n!$ elements and they all have the same probability under $\pp_{n,p}$. Moreover by the Lemma \ref{nl1} every orbit meets $B({\mathcal V}_{m,M},\epsilon)$, and  $B({\widetilde h},\eta)$ is invariant under $\sigma\in\Sigma$ and therefore under $\pi$.  Consequently
$$
B({\widetilde h},\eta)\subset\bigcup_\pi \pi^{-1}(B({\widetilde h},\eta)\cap B({\mathcal V}_{m,M},\epsilon))
$$
and the lemma follows.
\end{proof}
\begin{lmm}\label{nl3} There exists a function $\delta({\widetilde h}, \ep)$, depending only on ${\widetilde h}$ and $\ep$, with $\delta ({\widetilde h},\ep)\ra 0$ as $\ep\ra 0$,  such that for each arbitrary but fixed $\ep, m, M$ satisfying Lemma \ref{nl1},
$$
\lim_{\eta\ra 0} \limsup_{n\ra\infty}\frac{1}{n^2}\log\pp_{n,p}(B({\widetilde h},\eta)\cap B({\mathcal V}_{m,M},\epsilon) )\le - I_p({\widetilde h})+\delta({\widetilde h},\ep)$$
\end{lmm}
\begin{proof}
Since ${\mathcal V}_{m,M}$ is a finite union  $\bigcup_{m\le K\le M}{\mathcal V}_{K}$ it is sufficient to prove that  for each $K$
$$
\lim_{\eta\ra 0}\limsup_{n\ra\infty}\frac{1}{n^2}\log\pp_{n,p}(B({\widetilde h},\eta)\cap B({\mathcal V}_{K},\epsilon) )\le -I_p({\widetilde h})+\delta({\widetilde h},\epsilon)$$
and 
$\delta({\widetilde h}, \ep)\ra 0$ as $\ep\ra 0$. For fixed $K$, ${\mathcal V}_K$ consists of a compact set of functions in $L_1([0,1]^2)$ and can be covered by a 
finite number of spheres of radius $\ep$ in $L_1$ and therefore in $\mm$. It is therefore sufficient to show that for
fixed $K$ and $g\in {\mathcal V}_K$
$$
\lim_{\eta\ra 0} \limsup_{n\ra\infty}\frac{1}{n^2}\log\pp_{n,p}( B({\widetilde h},\eta)\cap B(g,2\epsilon) )\le -I_p({\widetilde h})+\delta({\widetilde h}, \epsilon)
$$
We can assume that $B({\widetilde h},\eta)\cap B(g,2\epsilon)\not= \emptyset$. Therefore $g\in B({\widetilde h},\eta+2\ep)$. Since $\eta\ra 0$ we can assume $\eta<\ep$ so that $g\in B({\widetilde h},3\ep)$. By lower semi-continuity of $I_p(\cdot)$, $I_p(f)\ge I_p({\widetilde h})-\delta( {\widetilde h}, \ep)$ on $B(g,2\epsilon)\subset B({\widetilde h}, 5\epsilon)$ 
and $\delta( {\widetilde h}, \ep)\ra 0$ as $\ep\ra 0$. We note that $B(g,2\epsilon)\subset \mm$ is weakly closed and therefore by the upper bound in Theorem \ref{weak}, 
\begin{align*}
\lim_{\eta\ra 0}&\limsup_{n\ra\infty}\frac{1}{n^2}\log\pp_{n,p}( B({\widetilde h},\eta)\cap B(g,2\epsilon) )\le \limsup_{n\ra\infty}\frac{1}{n^2}\log\pp_{n,p}( B(g,2\epsilon) )\\
& \le -\inf_{f\in B(g,2\epsilon)} I_p(f) \le -\inf_{f\in B({\widetilde h},5\epsilon)} I_p(f)\le -I_p({\widetilde h})+\delta({\widetilde h},\ep)
\end{align*}
where $\delta({\widetilde h},\ep)\ra 0$ as $\ep\ra 0$. \end{proof}
Lemma \ref{nl2} and Lemma \ref{nl3} yield \eqref{upperbound}, which proves the upper bound in Theorem \ref{main}.

\subsection{Proof of the lower bound in Theorem \ref{main}}
Let $h(x,y)\in\mm$ be given.  We  define 
$$p^{(n)}_{i,j}=n^2\int\int_{[\frac{i-1}{n},\frac{i}{n}]\times [\frac{j-1}{n},\frac{j}{n}]} h(x,y) dx dy$$
and the corresponding function $h_n(x,y)\in\mm$ by

$$
h_n(x,y)=\sum_{i,j} p^{(n)}_{i,j} {\bf 1}_{[\frac{i-1}{n},\frac{i}{n}]}(x){\bf 1}_{[\frac{j-1}{n},\frac{j}{n}]}(y).
$$
Since $\|h_n-h\|_{L_1([0,1]^2)}\to 0$, it follows that $d_\Box (h_n,h)\to 0$. It is therefore sufficient to prove that for any $\ep>0$ 
$$
\liminf_{n\ra\infty}\frac{1}{n^2}\log  \pp(d_\Box (f^{G(n,p)}, h_n)\le \ep)\ge - I_p(h).
$$
%or
%$$
%\liminf_{n\ra\infty}\frac{1}{n^2}\log  \pp_{n,p} (d_\Box (f, h_n)\le \ep)\ge - I_p(h)
%$$
We define an inhomogeneous random graph where the edge connecting  the vertices $i$ and $j$ is present  with probability $p^{(n)}_{i,j}$.  Different  edges are  independent.  If $\xi_{i,j}=1$ when the edge connecting $i,j$ is present and  $0$ otherwise then $\xi_{i,j}$ are independent Binomial random variables with $\pp(\xi_{i,j}=1)=p^{(n)}_{i,j}$. We denote by $\pp_{n,h}$ the measure on $\mm$ induced by
$$
f_n(x,y):=\sum_{i,j\atop i\not= j}\xi_{i,j}{\bf 1}_{[\frac{i-1}{n},\frac{i}{n}]}(x){\bf 1}_{[\frac{j-1}{n},\frac{j}{n}]}(y)
$$
If $A$ and $B$ are subsets of $\{1,\ldots, n\}$, it is straightforward to calculate 
\begin{align*}
\psi_n(\lambda)&:=\frac{1}{n^2}\log \ee^{\pp_{n,h}}\biggl[\exp\biggl( \lambda \sum_{i\in A, j\in B\atop i\not=j}(\xi_{i,j}-p^{(n)}_{i,j})\biggr)\biggr]\\ 
&=\frac{1}{n^2}\biggl[\sum_{i,j\in A\cap B\atop i> j}\log \ee^{\pp_{n,h}}\Big[\exp (2\lambda  (\xi_{i,j}-p^{(n)}_{i,j}))  \Big]\\
&\qquad +\sum_{\substack{i\in A\cap B,\, j\in B\cap A^c  \\ \text{ or } i\in A\cap B^c,\, j\in A\cap B \\ \text{ or }i\in A\cap B^c,\, j\in A^c\cap B}}\log \ee^{\pp_{n,h}}\Big[\exp (\lambda  (\xi_{i,j}-p^{(n)}_{i,j})) \Big]\biggr].
%\\
%&\qquad +\sum_{i\in A\cap B^c,j\in A\cap B}\log \ee^{\pp_{n,h}}\Big[\exp (\lambda  (\xi_{i,j}-p^{(n)}_{i,j})) \Big]\\
%&\qquad +\sum_{i\in A\cap B^c,j\in A^c\cap B}\log \ee^{\pp_{n,h}}\Big[\exp (\lambda  (\xi_{i,j}-p^{(n)}_{i,j})) \Big]\biggr]. 
\end{align*}
Each term in the sum is easily estimated by $\frac{\lambda^2}{2}$, providing an estimate of the form
$$
\pp_{n,h}\biggl(\biggl|\iint_{\widetilde{A}\times \widetilde{B}} (f_n-h_n) dxdy\biggr|\ge \ep\biggr)\le e^{-\frac{n^2\ep^2}{2}},
$$
where $\widetilde{A} = \cup_{i\in A} [\frac{i-1}{n}, \frac{i}{n}]$ and $\widetilde{B}$ is defined similarly. 
Since the number of sets like $\widetilde{A}\times \widetilde{B}$ is only $2^{2n}$ it follows that
$$
\pp_{n,h}(d_\Box(f_n,h_n)\ge \ep)\ra 0
$$
as $n\ra\infty$. Now the lower bound is easily established by a simple tilting argument. Denoting by $B_{\epsilon, n}$ the set
$\{f: d_\Box(f,h_n)\le \ep\}$
\begin{align*}
\pp_{n,p}(B_{\ep,n})&=\int_{B_{\ep,n}} d\pp_{n,p}=\int_{B_{\ep,n}} e^{-\log\frac{d\pp_{n,h}}{d\pp_{n,p}}}d\pp_{n,h}\\
&=\pp_{n,h}(B_{\ep,n})\frac{1}{\pp_{n,h}(B_{\ep,n})}\int_{B_{\ep,n}} e^{-\log\frac{d\pp_{n,h}}{d\pp_{n,p}}}d\pp_{n,h}.
\end{align*}
By Jensen's inequality
$$
\log \pp_{n,p}(B_{\ep,n})\ge \log \pp_{n,h}(B_{\ep,n})-\frac{1}{\pp_{n,h}(B_{\ep,n})}\int_{B_{\ep,n}} \log\frac{d\pp_{n,h}}{d\pp_{n,p}}d\pp_{n,h}.
$$
Since $\pp_{n,h}(B_{\ep,n})\ra 1$, it is easy to see that
$$
\liminf_{n\ra\infty}\frac{1}{n^2}\log \pp_{n,p}(B_{\ep,n})\ge -\lim_{n\ra\infty}\frac{1}{n^2}\int \log\frac{d\pp_{n,h}}{d\pp_{n,p}}d\pp_{n,h}
$$
The entropy cost of tilting (i.e.\ the integral in the preceding display) is
$$
\frac{1}{n^2}\sum_{i>j}\biggl(p^{(n)}_{i,j}\log\frac{p^{(n)}_{i,j}}{p}+(1-p^{(n)}_{i,j})\log\frac{1-p^{(n)}_{i,j}}{1-p}\biggr)
$$
which converges to $I_p(h)$ as $n\ra\infty$. This proves the lower bound.

\section{Conditional distributions}
Theorem \ref{main} gives estimates of the probabilities of rare events for $G(n,p)$. However, it does not quite answer the following question: given that some particular rare event has occurred, what does the graph look like? Naturally, one might expect that if $G(n,p)\in \WF$ for some closed set $\WF\subseteq\mmm$ satisfying 
\begin{equation}\label{inf0}
\inf_{\wh\in \WF^o} I_p(\wh) = \inf_{\wh\in \WF} I_p(\wh)> 0,
\end{equation}
then $G(n,p)$ should resemble one of the minimizers of $I_p$ in $\WF$. (Here $\WF^o$ denotes the interior of $\WF$, as usual.) In other words, given that $G(n,p)\in \WF$, one might expect that 
$\delta_\Box(G(n,p), \WF^*) \approx 0$, 
where $\WF^*$ is the set of minimizers of $I_p$ in $\WF$ and 
\[
\delta_\Box(G(n,p),\WF^*) := \inf_{\wh\in \WF^*} \delta_\Box(G(n,p), \wh). 
\]
However, it is not obvious that a minimizer must exist in $\WF$. Here is where  the compactness of $\mmm$ comes to the rescue yet one more time: since the function $I_p$ is lower semicontinuous on $\WF$ and $\WF$ is closed, therefore a minimizer must necessarily exist. The following theorem formalizes this argument.
\begin{thm}\label{conditional}
Take any $p\in (0,1)$. Let $\WF$ be a closed subset of $\mmm$ satisfying~\eqref{inf0}. Let $\WF^*$ be the subset of $\WF$ where $I_p$ is minimized. Then $\WF^*$ is non-empty and compact, and for each $n$,  and each $\ep >0$, 
\[
\pp(\delta_\Box(G(n,p), \WF^*) \ge \ep \mid G(n,p)\in \WF) \le e^{-C(\ep, \WF) n^2}
\]
where $C(\ep, \WF)$ is a positive constant depending only on $\ep$ and $\WF$. In particular, if $\WF^*$ contains only one element $\wh^*$, then the conditional distribution of $G(n,p)$ given $G(n,p)\in \WF$ converges to the point mass at $\wh^*$ as $n\ra\infty$. 
\end{thm}
\begin{proof}
Since $\mmm$ is compact and $\WF$ is a closed subset, therefore $\WF$ is also compact. Since $I_p$ is a lower semicontinuous function on $\WF$ (Lemma \ref{lower}) and $\WF$ is compact, it must attain its minimum on $\WF$. Thus, $\WF^*$ is non-empty. By the lower semicontinuity of $I_p$, $\WF^*$ is closed (and hence compact). Fix $\ep >0$ and let
\[
\WF_\ep := \{\wh \in \WF: \delta_\Box(\wh, \WF^*) \ge \ep\}.
\]
Then $\WF_\ep$ is again a closed subset. Observe that 
\begin{align*}
\pp(\delta_\Box(G(n,p), \WF^*) \ge \ep \mid G(n,p)\in \WF) &= \frac{\pp(G(n,p)\in \WF_\ep)}{\pp(G(n,p)\in \WF)}. 
\end{align*}
Thus, with 
\[
I_1 :=  \inf_{\wh\in \WF} I_p(\wh), \ \ I_2 := \inf_{\wh\in \WF_\ep} I_p(\wh),
\] 
Theorem \ref{main} and condition \eqref{inf0} give
\[
\limsup_{n\ra\infty} \frac{1}{n^2} \log \pp(\delta_\Box(G(n,p), \WF^*) \ge \ep \mid G(n,p)\in \WF) \le I_1-I_2. 
\]
The proof will be complete if it is shown that $I_1 < I_2$. 

Now clearly, $I_1\le I_2$. If $I_1=I_2$, the compactness of $\WF_\ep$ implies that there exists $\wh\in \WF_\ep$ satisfying $I_p(\wh)=I_2$. However, this means that $\wh\in \WF^*$ and hence $\WF_\ep \cap \WF^* \ne \emptyset$, which is impossible. 
\end{proof}

\section{Application to triangle counts}
\subsection{Brief history of the problem}
Let $T_{n,p}$ be the number of triangles in $G(n,p)$. The primary objective of this section is to compute the large deviation rate function for the upper tail of $T_{n,p}$ when $p$ remains fixed and $n \ra \infty$. In other words, given $p\in [0,1]$ and $\ep >0$, we wish to evaluate the limit
\begin{equation}\label{problem}
\lim_{n\ra\infty} \frac{1}{n^2}\log \pp(T_{n,p}\ge (1+\ep)\ee(T_{n,p}))
\end{equation}
as a function of $p$ and $\ep$. 

The problem of estimating tail probabilities like $\pp(T_{n,p}\ge (1+\ep)\ee(T_{n,p}))$ has been studied extensively in the random graphs literature, particularly in the case when $p$ is allowed to tend to zero as $n \ra \infty$. Computing upper and lower bounds on such tail probabilities that are sharp up to constants in the exponent was a prominent open problem in this area until until it was solved recently in~\cite{chatterjee10}. Let us refer to the  paper  \cite{chatterjee10} for a survey of the aforementioned literature. 

When $p$ is fixed, computing sharp upper and lower bounds is relatively easy. The difficult problem in this case is the exact evaluation of the limit~\eqref{problem}. The first progress in this direction was made in~\cite{chatterjeedey09} where it was shown that, given $p\in (0,1)$, there exist $p^3/6< t'\le t''< 1/6$ such that for all $t\in (p^3/6,t')\cup(t'',1/6)$, 
\begin{align}\label{cd}
\lim_{n\ra \infty}\frac{1}{n^2}\log \pp(T_{n,p} \ge tn^3) = - I_p((6t)^{1/3}), 
\end{align}
when $I_p$ is the entropy function defined in \eqref{ipdef1}.
Explicit formulas for $p'$ and $p''$ are also given in \cite{chatterjeedey09}. Unfortunately, the result does not cover all values of $(p,t)$; and neither is the above formula true for all $(p,t)$, as we shall see below.

There is a related unpublished manuscript by Bolthausen, Comets  and Dembo \cite{bolthausenetal09} on large deviations for subgraph counts. As of now, to the best of our knowledge, the authors of \cite{bolthausenetal09} have only looked at subgraphs that do not complete loops, like $2$-stars. Besides \cite{chatterjeedey09} and \cite{bolthausenetal09}, we know of no other papers that attack the exact evaluation of \eqref{problem} (or equivalently, \eqref{cd}).

\subsection{Exact large deviations for the upper tail} In this subsection, the limit~\eqref{cd} is evaluated for all values of $p$ and $t$. It comes as the solution of the following  variational problem. Let $\mm$, $\mmm$ and $\delta_\Box$ be defined as in Section \ref{intro}.  For each $f\in \mm$, let
\[
T(f) := \frac{1}{6} \int_0^1\int_0^1\int_0^1 f(x,y) f(y,z) f(z,x) \;dx\;dy\;dz
\]
and let $I_p(f)$ be defined as in \eqref{ipdef2}. Note that $T$ can be defined on $\mmm$ simply by letting $T(\wf) := T(f)$, because $T$ is a continuous map on $\mm$  under the $d_\Box$ pseudometric (Theorem 3.7 in \cite{borgsetal08}). 

For each $p\in [0,1]$ and $t\in [0,1/6)$, let
\begin{equation}\label{phidef}
\phi(p,t) := \inf\{I_p(f) : f\in \mm,\;T(f) \ge t\}. 
\end{equation}
For $t \ge 1/6$, let $\phi(p,t)=\infty$. The following result gives the large deviation rate function for the upper tail of $T_{n,p}$. (Note that this is just an illustrative example. Theorem \ref{main} can be used to derive large deviations for any subgraph count, or even joint large deviations for the counts of more than one subgraph.)
\begin{thm}\label{triangle}
Let $G(n,p)$ be the Erd\H{o}s-R\'enyi random graph on $n$ vertices with edge probability $p$. Let $T_{n,p}$ denote the number of triangles in $G(n,p)$. Let $\phi$ be defined as above. Then for each $p\in (0,1)$ and each $t\ge 0$, 
\begin{align*}
\lim_{n\ra\infty} \frac{1}{n^2}\log \pp(T_{n,p}\ge tn^3) =  -\phi(p,t). 
\end{align*}
Next, take any $p\in (0,1)$ and $t\in (p^3/6, 1/6)$. Let $F_{p,t}^*$ be the set of minimizers for the variational problem \eqref{phidef} and $\WF_{p,t}^*$ be its image in $\mmm$. Then $\WF_{p,t}^*$ is a non-empty compact subset of $\mmm$. Moreover, for each $\ep >0$ there exists a positive constant $C(\ep, p,t)$ depending only on $\ep$, $p$ and $t$ such that for any~$n$,
\[
\pp(\delta_\Box(G(n,p), \WF_{p,t}^*) \ge \ep\mid T_{n,p}\ge tn^3) \le e^{-C(\ep,p,t)n^2}. 
\]
\end{thm}
\begin{proof}
Let $F := \{f\in \mm: T(f)\ge t\}$. By Theorem 3.7 in \cite{borgsetal08}, $F$ is a closed subset of $\mm$. Therefore by Theorem \ref{main}, 
\begin{align*}
\limsup_{n\ra\infty} \frac{1}{n^2}\log \pp(T_{n,p}\ge tn^3) &= \limsup_{n\ra\infty} \frac{1}{n^2}\log \pp(G(n,p)\in F)\\
&\le -\inf_{h\in F} I_p(h) = -\phi(p,t). 
\end{align*}
Next, let $U := \{f\in \mm: T(f)> t\}$. Again by Theorem 3.7 of \cite{borgsetal08}, $U$ is an open set. Therefore by Theorem \ref{main}, for each $\ep >0$,
\begin{align*}
\liminf_{n\ra\infty} \frac{1}{n^2}\log \pp(T_{n,p}\ge tn^3) &\ge \liminf_{n\ra\infty} \frac{1}{n^2}\log \pp(G(n,p)\in U)\\
&\ge -\inf_{h\in U} I_p(h) \ge -\phi(p,t+\ep). 
\end{align*}
In Proposition \ref{properties} below, it is proved that $\phi$ is a continuous function of $t$ for every fixed $p$. This completes the proof of the first assertion of the theorem. The second assertion is merely a corollary of Theorem \ref{conditional}. The condition \eqref{inf0} required for Theorem~\ref{conditional} can be easily shown to follow from the continuity of $\phi$ in $t$, because any $\wf$ with $T(\wf) >t$ lies in the interior of the set $\{\wh: T(\wh)\ge t\}$. 
\end{proof}

\subsection{Properties of the rate function}
Given Theorem \ref{main}, there is a natural desire to understand the rate function~$\phi$. The following proposition summarizes some basic properties of $\phi$. The first property is required in the proof of Theorem \ref{triangle} above.
\begin{prop}\label{properties}
For each fixed $p\in (0,1)$, the following hold:
\begin{enumerate}[\textup{(}i\textup{)}]
\item The function $\phi(p,t)$ is continuous in $t$ in the interval $[0,1/6)$. 
\item As a function of $t$, $\phi(p,t) = 0$ in the interval $[0, p^3/6]$ and strictly increasing in $(p^3/6,1/6)$. Moreover, for $p^3/6< t< s< 1/6$, 
\[
\phi(p,t)< (t/s)^{1/3} \phi(p,s).
\] 
\item For $t\in (p^3/6, 1/6)$, $\phi(p,t)$ can be alternately represented as
\[
\phi(p,t) := \inf\{I_p(f) : f\in \mm,\;T(f) = t\}. 
\]
Moreover, if $\{f_n\}_{n\ge 1}$ is a sequence in $\mm$ such that $T(f_n)\ge t$ for all $n$ and $I_p(f_n) \ra \phi(p,t)$, then $T(f_n)\ra t$. In particular, the elements of $\WF^*_{p,t}$ all satisfy $T(f)=t$. 
\end{enumerate}
\end{prop}
\begin{proof}
Let us start by proving that $\phi$ is continuous in $t$. 
For each $f\in\mm$ and $\delta \in [0,1]$, let 
\[
f^\delta := f + \delta(1-f).
\]
By the inequality
\begin{align*}
&(a+\delta(1-a))(b+\delta(1-b))(c+ \delta(1-c)) \\
&= abc + \delta((1-a)bc + (1-b)ac+(1-c)ab)  \\
&\qquad + \delta^2((1-a)(1-b)c + (1-a)(1-c)b + (1-b)(1-c)a) \\
&\qquad + \delta^3(1-a)(1-b)(1-c) \\
&\ge abc + \delta^3((1-a)bc + (1-b)ac+(1-c)ab)  \\
&\qquad + \delta^3((1-a)(1-b)c + (1-a)(1-c)b + (1-b)(1-c)a) \\
&\qquad + \delta^3(1-a)(1-b)(1-c) \\
&= abc + \delta^3(1-abc),
\end{align*}
we see that 
\begin{equation}\label{tineq}
T(f^\delta) \ge T(f)(1-\delta^3) + \frac{\delta^3}{6}. 
\end{equation}
Take any $t\in [0,1/6)$ and any $f$ such that $T(f) \ge t$. Suppose $t_n \downarrow t$. Let $\delta_n$ be the smallest number in $[0,1]$ such that $T(f^{\delta_n}) \ge t_n$. By \eqref{tineq} it follows that $\delta_n$ exists and $\lim_{n\ra \infty} \delta_n = 0$. Therefore by the dominated convergence theorem, $\lim_{n\ra \infty} I_p(f^{\delta_n}) = I_p(f)$. Thus,
\begin{align*}
\lim_{n\ra \infty} \phi(p,t_n) \le \lim_{n\ra \infty} I_p(f^{\delta_n}) = I_p(f). 
\end{align*}
Since this is true for every $f$ such that $T(f) \ge t$ and $\phi$ is a non-decreasing function in $t$, this proves the right continuity of $\phi$. 

Next, take a sequence $t_n \uparrow t$. Let $f_n$ be a sequence of functions such that $T(f_n) \ge t_n$ and 
\[
\lim_{n\ra \infty} I_p(f_n) = \lim_{n\ra\infty} \phi(p,t_n).
\]
For each $n$, let $\delta_n$ be the smallest number in $(0,1)$ such that $T(f_n^{\delta_n}) \ge t$. By~\eqref{tineq}, $\delta_n$ exists and $\lim_{n\ra \infty}\delta_n = 0$. Now, the function $I_p$ on $[0,1]$ (defined in \eqref{ipdef1}) is uniformly continuous on $[0,1]$. As a consequence, 
\[
\lim_{\delta \ra 0} \sup_{f\in \mm} |I_p(f^\delta) - I_p(f)| = 0.
\]
In particular, 
\[
\lim_{n\ra \infty} I_p(f_n^{\delta_n}) = \lim_{n\ra \infty} \phi(p,t_n). 
\]
But $\phi(p,t) \le I_p(f^{\delta_n}_n)$ for each $n$. By the monotonicity of $\phi$, this proves left continuity.

Next, note that since $\pp(T(G_{n,p}) \ge tn^3) \ra 1$ for any $t< p^3/6$, Theorem~\ref{triangle} and the continuity of $\phi$ imply that $\phi(p,t) = 0$ for $t\le p^3/6$. Let us now show that $\phi(p,t)$ is strictly increasing in $t$ when $t\in (p^3/6, 1/6)$. 

%Without loss of generality, we can assume that $f(x,y) \ge p$ for all $x,y$, because otherwise, the function $g(x,y) := \max\{f(x,y), p\}$ satisfies the above criteria. 

Fix $p^3/6\le t < s < 1/6$. Fix $\ep > 0$. Take any $f\in \mm$ such that $T(f) \ge s$. For each $\delta \in (0,1)$, let 
\begin{equation}\label{fdelta}
f_\delta (x,y) := (1-\delta) f(x,y) + \delta p. 
\end{equation}
By the inequality
\begin{align*}
((1-\delta)a+\delta p)((1-\delta) b+\delta p)((1-\delta)c+ \delta p ) \ge (1-\delta)^3 abc
\end{align*}
we have 
\begin{equation*}
T(f_{\delta})\ge (1-\delta)^3 T(f)\ge (1-\delta)^3 s. 
\end{equation*}
Thus, if we take $\delta$ such that $(1-\delta)^3 = t/s$, then  $T(f_\delta) \ge t$. By the convexity of the function $I_p$ on $[0,1]$ defined in \eqref{ipdef1}, we see that for any $x\ge p$,
\[
I_p((1-\delta) x+ \delta p) \le (1-\delta) I_p(x)+ \delta I_p(p)= (1-\delta) I_p(x),
\]
and therefore,
\begin{equation}\label{tfineq}
I_p(f_\delta) \le (1-\delta) I_p(f). 
\end{equation}
Thus, $\phi(p,t) \le I_p(f_\delta) \le (t/s)^{1/3}I_p(f)$.
Since this holds for any $f$ with $T(f)\ge s$, it shows that 
\begin{equation}\label{phist}
\phi(p, t)\le (t/s)^{1/3} \phi(p,s).
\end{equation} 
To show $\phi$ is a strictly increasing function of $t$ in the interval $(p^3/6, 1/6)$, it therefore suffices to prove that $\phi(p,t) > 0$ for $t$ in this interval. This follows easily, since the strict convexity of $I_p$ on $[0,1]$ and equation \eqref{tfineq} show that equality in \eqref{phist} can hold only if $f\equiv p$ almost everywhere for some $f$ such that $T(f)\ge s$, which is impossible since $s> p^3/6$. 

Next, fix $t\in (p^3/6, 1/6)$ and take any sequence $\{f^{(n)}\}$ in $\mm$ such that $T(f^{(n)})\ge t$ for all $n$ and $I_p(f^{(n)})\ra\phi(p,t)$. Recall the subscript notation introduced in \eqref{fdelta} above. By the continuous mapping theorem, there exist $\delta_n\in [0,1]$ such that for each $n$,
\[
T(f_{\delta_n}^{(n)}) = t. 
\]
Therefore by \eqref{tfineq},
\[
\phi(p,t) \le I_p(f^{(n)}_{\delta_n}) \le (1-\delta_n)I_p(f^{(n)}), 
\]
which proves that $\delta_n \ra 0$. This proves that
\[
\phi(p,t) = \inf\{I_p(f): f\in \mm, \; T(f) = t\}. 
\]
Since $\delta_n \ra 0$, this also proves that $T(f^{(n)}) \ra t$. 
This completes the proof of Proposition \ref{properties}.
\end{proof}

\subsection{The `Replica Symmetric' phase}
Note that there are two ``extreme'' functions that satisfy $T(f) = t$. First, there is the constant function 
\begin{equation}\label{ct}
c_t(x,y) \equiv (6t)^{1/3}.
\end{equation}
%Graph theoretically, this represents a limiting Erd\H{o}s-R\'enyi graph $G(n,r)$ in the sense of Lovasz and Szegedy \cite{lovaszszegedy06}, where $r= (6t)^{1/3}$.
On the opposite extreme, there is the function $\chi_t$, defined as 
\begin{equation}\label{chit}
\chi_t(x,y) := 
\begin{cases}
1 &\text{ if } \max\{x,y\}\le (6t)^{1/3},\\
0 &\text{ otherwise.}
\end{cases}
\end{equation}
%Graph theoretically, this represents a graph with a single clique.
In a limiting sense, $c_t$ represents an Erd\H{o}s-R\'enyi random graph with edge probability $(6t)^{1/3}$, while $\chi_t$ represents the union of a clique of size $n(6t)^{1/3}$ and a set of isolated vertices of size $n(1-(6t)^{1/3})$. 

It is simple to see that \eqref{cd} holds if and only if the infimum in \eqref{phidef} is attained at the constant function $c_t$. The following theorem gives a sufficient condition for this to happen. This extends the main result of \cite{chatterjeedey09}. The methods of \cite{chatterjeedey09} are closely related to methods from statistical physics; drawing inspiration from this connection, one may call the region where $c_t$ solves the variational problem \eqref{phidef} as the `replica symmetric phase' of the problem. 
\begin{thm}\label{symm}
For each $0< p< 1$ and $t\in (p^3/6, 1/6)$, let 
$h_p(t) := I_p((6t)^{1/3})$.
For $t\in [0,p^3/6]$, let $h_p(t)=0$. Let $\hat{h}_p$ be the convex minorant of $h_p$ (i.e.\ the maximum convex function lying below $h_p$). If $t$ is a point in $(p^3/6, 1/6)$ where $h_p(t) = \hat{h}_p(t),$ then the variational problem \eqref{phidef} for this pair $(p,t)$ is uniquely solved by the constant function $c_t \equiv (6t)^{1/3}$. Consequently, $\phi(p,t) = h_p(t)$. Moreover, for such $(p,t)$, for each $\ep >0$
\[
\lim_{n\ra\infty}\pp(\delta_\Box(G(n,p), c_t) \ge\ep \mid T_{n,p}\ge tn^3) = 0.
\]
Since $c_t$ is the limit of $G(n,(6t)^{1/3})$, this means that for such $(p,t)$ the conditional distribution of $G(n,p)$ given $T_{n,p}\ge tn^3$ is indistinguishable from the law of $G(n, (6t)^{1/3})$ in the large $n$ limit.
%the conditional distribution of $G(n,p)$ given $T_{n,p}\ge tn^3$ converges to the point mass at $c_t$ as $n\ra\infty$.  
\end{thm}
\begin{proof}
Since $h_p$ is an increasing function, it is easy to see that $h_p(t) = \hat{h}_p(t)$ for $t\le p^3/6$. (Incidentally, this also shows that $\hat{h}_p$ is an increasing function in $[0,1/6)$, and strictly increasing in $(p^3/6, 1/6)$.) %So there is nothing to prove when $t\le p^3/6$.

Suppose $t$ is a point in $(p^3/6, 1/6)$ such that $h_p(t) = \hat{h}_p(t)$. We claim that there exists $\beta > 0$ such that 
\begin{equation}\label{betadef}
t = \argmax_{x\in [0,1/6]} (\beta x - h_p(x)). 
\end{equation}
To see this, observe that since $\hat{h}_p$ is convex and strictly increasing in the interval $(p^3/6,1/6)$, there exists $\beta> 0$ and $c\in \rr$ such that the line $y = \beta x + c$ lies below the curve $y = \hat{h}_p(x)$ and touches it at $x=t$. But we also know that $h_p$ lies above $\hat{h}_p$ and the two curves touch at $t$. Thus,
\[
t = \argmax_x (\beta x + c - h_p(x)) = \argmax_x (\beta x - h_p(x)).
\]
This proves the claim. 

Now take any $f\in \mm$ such that $T(f)\ge t$. Let $c_t$ be the function that is identically equal to $(6t)^{1/3}$. Then $T(c_t)= t$, and by H\"older's inequality and~\eqref{betadef},
\begin{align*}
&\beta t - I_p(c_t) = \iint \biggl(\frac{\beta c_t(x,y)^3}{6} - I_p(c_t(x,y))\biggr) \; dx\; dy\\
&\ge \iint \biggl(\frac{\beta f(x,y)^3}{6} - I_p(f(x,y))\biggr) \; dx\; dy\\
&\ge \frac{\beta}{6}\iiint f(x,y) f(y,z)f(z,x) \;dx\;dy\;dz - \iint I_p(f(x,y))\;dx\;dy\\
&\ge \beta t - I_p(f). 
\end{align*}
Thus, $c_t$ minimizes $I_p(f)$ among all $f$ such that $T(f)\ge t$. This shows that
\[
\phi(p,t) = I_p(c_t) = h_p(t). 
\]
The uniqueness of the optimizer follows from the H\"older step in the above deduction. Finally, the claim about the conditional distribution follows from Theorem \ref{conditional} and the uniqueness of the minimizer.  
%This completes the proof of Theorem \ref{symm}.
\end{proof}
It is easy to show that for any $p > 0$, $h_p(t)=\hat{h}_p(t)$ for all $t\in (p^3/6, t') \cup(t'', 1/6)$ where $t'$ and $t''$ depend on $p$. Similarly, given any $t\in (0,1/6)$, there exists $p'< (6t)^{1/3}$ depending on $t$ such that for all $p\in (p', (6t)^{1/3})$, $h_p(t) = \hat{h}_p(t)$. Thus, there is a nontrivial set of $(p,t)$ where $c_t$ solves the variational problem and consequently $\phi(p,t) = h_p(t)$. As mentioned before, this recovers the main result of~\cite{chatterjeedey09}. The conclusion about the conditional distribution is a new result. 

\subsection{Replica Symmetry Breaking} 
Given Theorem \ref{symm}, it is quite interesting to note that the variational problem \eqref{phidef} is not solved by constant functions  everywhere. From the physical point of view espoused in \cite{chatterjeedey09}, however, this is not surprising; it is simply the effect of replica symmetry breaking down in the `low temperature regime'.

The phase transition is very easy to establish using Theorem \ref{triangle}, by comparing the performances of $c_t$ and $\chi_t$ defined in \eqref{ct} and \eqref{chit}. A simple computation shows that for any $t\in (0,1/6)$, 
\begin{align}\label{nonopt}
\lim_{p\ra 0} \frac{I_p(c_t)}{\log (1/p)} &= \frac{(6t)^{1/3}}{2}> \frac{(6t)^{2/3}}{2} = \lim_{p\ra 0} \frac{I_p(\chi_t)}{\log (1/p)}. 
\end{align}
Combining the above observation with Theorem \ref{conditional}, it follows easily that there are values of $(p,t)$ such that given $T_{n,p}\ge tn^3$, the graph $G(n,p)$ must look different than an Erd\H{o}s-R\'enyi graph. Again, compactness is crucial.  This is formalized by the following theorem. 
\begin{thm}\label{breaking}
Let ${\widetilde C}$ denote the set of constant functions in $\mmm$. For each $t\in (0,1/6)$, there exists $p'>0$ such that for all $p< p'$, the variational problem \eqref{phidef} is not solved by the constant function $c_t\equiv (6t)^{1/3}$. Moreover, for such $(p,t)$, $\delta_\Box(\WF^*_{p,t}, {\widetilde C}) > 0$, where $\WF^*_{p,t}$ is the set of minimizers defined in Theorem \ref{triangle}. Consequently, there exists $\ep >0$ such that 
$$\lim_{n\ra\infty} \pp(\delta_\Box(G(n,p), {\widetilde C}) > \ep\mid T_{n,p}\ge tn^3) = 1.$$ 
\end{thm}
\begin{proof}
By \eqref{nonopt} we see that for each $t$, there exists $p' >0$ such that $c_t$ is not a minimizer for the problem \eqref{phidef} if $p < p'$. Take any such $(p,t)$. By Theorem~\ref{triangle}, $\WF_{p,t}^*$ is non-empty and compact. By part $(iii)$ of Proposition~\ref{properties} and the non-optimality of $c_t$, it follows that ${\widetilde C}$ and $\WF_{p,t}^*$ must be disjoint. But ${\widetilde C}$ and $\WF_{p,t}^*$ are both compact subsets of $\mmm$. Therefore, $\delta_\Box({\widetilde C}, \WF_{p,t}^*) >0$. The last claim follows by Theorem \ref{triangle}. 
\end{proof}

\subsection{The double phase transition}
A combination of Theorem \ref{breaking} and Theorem \ref{symm} shows that for all small enough $p$, the variational problem \eqref{phidef} has a `double phase transition'. (Actually, there may be more than two phase transitions, but we show that there is at least two.) 

Indeed, for all small enough $p$, Theorem \ref{symm} (or the results of \cite{chatterjeedey09}) show that there exists $p^3/6 < t'\le t''< 1/6$ such that $\phi(p,t) = I_p((6t)^{1/3})$ for all $t\in (p^3/6, t') \cup (t'', 1/6)$. On the other hand by Theorem \ref{breaking}, it follows that for all small enough $p$, the variational problem \eqref{phidef} is not solved by a constant function at the point $(p, 1/2)$. Combining these two observations gives the following theorem.
\begin{thm}
There exists $p_0 >0$ such that if $p\le p_0$, then there exists $p^3/6< t'< t''<1/6$ such that the variational problem \eqref{phidef} is solved by the  constant function $c_t \equiv (6t)^{1/3}$ when $t\in (p^3/6, t')\cup (t'', 1/6)$, but there is a non-empty subset of $(t', t'')$ where all optimizers are non-constant.
\end{thm}
Of course, as shown by Theorems \ref{symm} and \ref{breaking}, the significance of optimizers being constant or non-constant is in whether the conditional behavior of $G(n,p)$ given $T_{n,p}\ge tn^3$ is close to that of an Erd\H{o}s-R\'enyi graph or not. 

\subsection{The small $p$ limit}
The last theorem of this paper describes the nature of $\phi(p,t)$ and $\WF^*_{p,t}$ when $t$ is fixed and $p$ is very small, tending to zero. The essence of the result, perhaps not surprisingly, is that when $t$ is fixed and $p\ra 0$, then conditionally on the event $\{T_{n,p}\ge tn^3\}$ the graph $G(n,p)$ must look like a clique.
\begin{thm}\label{clique}
For each $t\in (0,1/6)$, 
\[
\lim_{p\ra 0} \frac{\phi(p,t)}{\log (1/p)} = \frac{(6t)^{2/3}}{2}. 
\]
Moreover, if $\chi_t$ is the function defined in \eqref{chit} and $\WF^*_{p,t}$ is defined as in Theorem \ref{triangle}, then the set $\WF^*_{p,t}$ converges to the point ${\widetilde \chi}_t$ as $p\ra0$, in the sense that 
\[
\lim_{p\ra 0} \sup_{\wf\in\WF^*_{p,t}}\delta_\Box(\wf, {\widetilde\chi}_t)=0. 
\]
Consequently, for each $\ep >0$,
\[
\lim_{p\ra0} \lim_{n\ra\infty} \pp(\delta_\Box(G(n,p), {\widetilde\chi}_t) \ge \ep \mid T_{n,p}\ge tn^3) = 0. 
\]
\end{thm} 
%An interesting consequence of Theorem \ref{clique} is that there exists $p$ such that the rate function $\phi(p, \cdot)$ has nonconvex portions. In the nonconvex regions, $\phi(p,t)$ cannot be represented via Fenchel-Legendre transforms. 
\begin{proof}
In this proof, $C$ will denote any constant that does not depend on anything else. The value of $C$ may change from line to line. All integrals will be over the interval $[0,1]$. 

Fix $t\in (0,1/6)$. For each $p< (6t)^{1/3}$, choose a function $f_p\in F^*_{p,t}$. From the definition of $I_p$, observe that if $p\le 1/2$, 
\begin{align}\label{ipfp}
\biggl|I_p(f_p) - \frac{1}{2}\iint f_p(x,y) \log \frac{1}{p} \,dx\,dy \biggr|&\le C.
\end{align}
On the other hand, by the definition of $\WF^*_{p,t}$, 
\begin{align*}
I_p(f_p) &\le I_p(\chi_t) \le \frac{(6t)^{2/3}}{2}\log \frac{1}{p}+ C. 
\end{align*}
Combining the last two inequalities gives
\begin{equation}\label{f2}
\iint f_p(x,y) \, dx\, dy \le (6t)^{2/3} + \frac{C}{\log(1/p)}. 
\end{equation}
%Thus, 
%\begin{equation}\label{f2}
%\biggl(\iint f_p(x,y)\, dx\, dy\biggr)^3 \le \biggl((6t)^{2/3} + \frac{C}{\log (1/p)}\biggr)^{3}\le (6t)^2 + \frac{C}{\log (1/p)}. 
%\end{equation}
Next, let 
\[
h_p(x,y) := \int f_p(x,z) f_p(z,y) \, dz. 
\]
Then by two applications of the Cauchy-Schwarz inequality, the inequality~\eqref{f2}, and the fact that $f_p(x,y)\in [0,1]$, we get the following important sequence of inequalities. 
\begin{equation}\label{f3}
\left\{
\begin{split}
(6t)^2 &= (6 T(f_p))^2 = \biggl(\iint h_p(x,y) f_p(x,y) \,dx\, dy\biggr)^2\\
&\le \iint h_p^2(x,y) \,dx \, dy \iint f_p^2(x,y) \, dx\, dy\\
&\le \int \biggl(\int f_p^2(x,z) \,dz \int f_p^2(y,z) \, dz\biggr) \, dx\, dy\iint f_p^2(x,y) \, dx \, dy\\
&= \biggl(\iint f_p^2(x,y) \, dx\, dy\biggr)^3 \\
&\le \biggl(\iint f_p(x,y) \, dx\, dy\biggr)^3\le \biggl((6t)^{2/3} + \frac{C}{\log (1/p)}\biggr)^{3}. 
\end{split}
\right.
\end{equation}
A direct consequence of \eqref{f3}, combined with \eqref{ipfp}, is that 
\[
\lim_{p\ra0} \frac{\phi(p,t)}{\log(1/p)} = \lim_{p\ra0} \frac{I_p(f_p)}{\log(1/p)} = \lim_{p\ra0} \frac{1}{2}\iint f_p(x,y) \, dx\, dy = \frac{(6t)^{2/3}}{2},
\]
which proves the first assertion of the theorem. A second important consequence of \eqref{f3}, to be useful later, is that 
\begin{equation}\label{f1f}
\iint f_p(x,y) (1-f_p(x,y)) \, dx \, dy \le \frac{C}{\log (1/p)}. 
\end{equation}
Next, note that for any $x,y$, 
\begin{align*}
&\frac{1}{2}\iint (f_p(x,z)f_p(z', y) - f_p(x, z')f_p(z, y))^2 \, dz\, dz'\\
&= \iint f_p^2(x,z) f_p^2(z', y)\, dz\,dz' - h_p^2(x,y)\\
&= \int f_p^2 (x,z) dz \int f_p^2(y,z) dz - h_p^2(x,y).  
\end{align*}
It follows from this and \eqref{f3} that 
\begin{align*}
&\iiiint (f_p(x,z)f_p(z', y) - f_p(x, z')f_p(z, y))^2 \, dz\, dz'\, dx\, dy\iint f_p^2(x,y) \, dx\, dy\\
&\le \frac{C}{\log (1/p)}. 
\end{align*}
The above inequality and the lower bound on $\iint f_p^2(x,y) dxdy$ from \eqref{f3} give
\begin{equation}\label{i4}
\begin{split}
&\iiiint (f_p(x,z)f_p (z', y) - f_p(x, z')f_p(z, y))^2 \, dz\, dz'\, dx\, dy \\
&\le \frac{C}{t^{2/3}\log (1/p)}. 
\end{split}
\end{equation}
Let $M_p := \iint f_p(x,y) dxdy$. 
For each $x$, let $m_p(x) := M_p^{-1/2}\int f_p(x,y) dy$.  An application of Jensen's inequality to \eqref{i4} gives
\begin{align*}
\iint (f_p(x,z)M_p - M_pm_p(x) m_p(z))^2 \, dz\, dx &\le \frac{C}{t^{2/3}\log (1/p)}.
\end{align*}
By \eqref{f3}, $M_p \ge (6t)^{2/3}$. Thus, $m_p$ is bounded by $(6t)^{-1/3}$, and 
\begin{equation}\label{fm}
\iint (f_p(x,z) - m_p(x) m_p(z))^2 \, dz\, dx \le \frac{C}{t^{2}\log (1/p)}.
\end{equation}
For each $p$, let $n_p:[0,1]\ra[0,(6t)^{-1/3}]$ be a step function (i.e.\ a function that is constant on intervals) such that 
\begin{equation}\label{mpnp}
\int (m_p(x)-n_p(x))^2 dx \le p. 
\end{equation}
(Such functions exist because we can approximate $m_p$ by a continuous function to any degree of accuracy by Lusin's Theorem \cite{rudin87}, and then approximate the continuous function by a step function.)

Let $\sigma_p$ be a measure preserving bijection of $[0,1]$ such that $n_p(\sigma_p x)$ is a non-increasing function. Such a bijection is easy to construct because $n_p$ is a step function. Let $\ell_p(x) := n_p(\sigma_p x)$ and $g_p(x,y) := f_p(\sigma_p x, \sigma_p y)$. By the monotonicity and uniform boundedness of $\ell_p$ there exists a sequence $\{p_i\}_{i\ge 1}$ decreasing to zero such that $\ell_{p_i}$ converges in $L^2$ to a limit function~$\ell$. Therefore by \eqref{fm} and \eqref{mpnp}, $g_{p_i} \ra g$ in $L^2$, where 
\[
g(x,y) :=\ell(x)\ell(y). 
\]
By this and \eqref{f1f}, $g$ is a $\{0,1\}$-valued function. It is not difficult to see from this and the non-negativity of $\ell$ that $\ell$ must also be $\{0,1\}$-valued. (If $\ell(x)\not\in \{0,1\}$ on some set of positive measure, then there may be a set $A$ of positive measure where $\ell(x)\in (0,1)$, or there may be a set $A$ of positive measure where $\ell(x)\in (1,\infty)$. In either case, $g(x,y) \not \in \{0,1\}$ on $A\times A$.)  Since $\ell$ is monotone decreasing, it follows that $\ell$ must be the indicator of an interval of the form $[0,b]$ for some $b\in [0,1]$. Lastly, \eqref{f3} implies that $\iint g(x,y) dxdy = (6t)^{2/3}$, and therefore $b = (6t)^{1/3}$. Consequently, $g= \chi_t$.

The above argument establishes that for any collection $\{f_p\}_{p>0}$ such that $f_p\in F^*_{p,t}$ for each $p$, there is a sequence $\{p_i\}_{i\ge 1}$ decreasing to zero such that $f_{p_i}\ra \chi_t$ in the cut metric. The same argument can be extended to show that for any sequence $\{f_{p_i}\}$ such that $p_i\ra 0$ and $f_{p_i} \in F^*_{p_i,t}$ for each $i$, there is a subsequence converging to $\chi_t$ in the cut metric. This proves the second assertion of the theorem. The last claim of the theorem follows from this and Theorem~\ref{triangle}. 
\end{proof}

\subsection{Open questions}
There are many questions that remain unresolved, even in the simple example of upper tails for triangle counts that has been  analyzed in this section. For instance, what is the set of optimal solutions of the variational problem \eqref{phidef} in the broken replica symmetry phase (i.e.\ where the optimizer is not a constant)? Is the solution unique in the quotient space $\mmm$, or can there exist multiple solutions? Is it possible to explicitly compute a nontrivial solution of \eqref{phidef} for at least some value of $(p,t)$? Is it possible to even numerically evaluate or approximate a solution using a computer? Does Theorem \ref{symm} characterize the full replica symmetric phase? If not, what is the exact phase transition boundary? What happens in the sparse case where $p$ and $t$ are both allowed to tend to zero? At the time of writing this paper, we do not know how to answer any of these questions. 

\vskip.2in
\noindent{\bf Acknowledgments.} The authors thank Amir Dembo for his crucial role in motivating this research and suggesting the formula for the rate function in the case of subgraph counts, Joel Spencer for hinting  that Szemer\'edi's lemma may be useful, and Persi Diaconis for pointing us to the work of Lov\'asz and coauthors (which helped in cutting down the size of the paper by a half). The authors also thank the referee for a very careful report.

\end{document}